\documentclass{amsart}

\usepackage{amssymb,amsthm}
\usepackage{enumitem}
\usepackage{graphicx}
\usepackage{mathptmx}
\usepackage{verbatim}

\newtheorem{conjecture}{Conjecture}
\newtheorem{corollary}{Corollary}
\newtheorem{definition}{Definition}
\newtheorem{example}{Example}
\newtheorem{lemma}{Lemma}
\newtheorem{proposition}{Proposition}
\newtheorem{theorem}{Theorem}

\begin{document}

\title{Coherence of countably many bets}

\author{Rafael B. Stern}
\address{Department of Statistics \\
	 Carnegie Mellon University \\ 
	 Pittsburgh, PA 15217, USA}
\email{rbstern@gmail.com}
\thanks{We thank Georg Goerg, Jessi Cisewski, Julio Stern, Mark Schervish, Rafael Izbicki and Teddy Seidenfeld for their valuable comments.}
	 
\author{Joseph B. Kadane}
\address{Department of Statistics \\
	 Carnegie Mellon University \\ 
	 Pittsburgh, PA 15217, USA}
\email{kadane@stat.cmu.edu}


\begin{abstract}
De Finetti's betting argument is used to justify finitely additive probabilities when only finitely many bets are considered. Under what circumstances can countably many bets be used to justify countable additivity? In this framework, one faces issues such as the convergence of the returns of the bet. Generalizations of de Finetti's \cite{deFinetti2} argument depend on what type of conditions on convergence are required of the bets under consideration. Two new such conditions are compared with others presented in the literature.
\end{abstract}

\maketitle

\section{Betting Systems}
Betting systems have been an object of interest to probabilists and statisticians dating back to the genesis of the disciplines. In particular, \cite{deFinetti2} considers relations between the axioms of probability and an experiment involving bets. Let $\Omega$ be the sample space. Consider that, for each event $A \subset \Omega$, a bet on $A$ is such that \textit{you} get $\$1$ from the broker if $A$ is observed and $0$, otherwise. Also, for each $A$ in an arbitrary collection of events $\mathcal{C} \subset \mathcal{P}(\Omega)$, $P(A)$ is the price for which you would be willing to either buy or sell such a bet. The broker can then require you to buy or sell any finite quantity of a finite number of bets on these events using your prices. You incur a sure loss if you always lose money, no matter what events are observed. If it is impossible for you to incur a sure loss, the assignment of prices is coherent. De Finetti \cite{deFinetti2} shows that, if $\mathcal{C}$ is a field over $\Omega$, $P$ is coherent 
iff for every $A \in \mathcal{C}$, $0 \leq P(A) \leq 1$, $P(\Omega)=1$ and $P$ is finitely additive, that is, for every $A \in \mathcal{C}$ and $B \in \mathcal{C}$ such that $A \cup B \in \mathcal{C}$, $P(A \cup B) = P(A)+P(B)$ . From now on, we call such prices finitely additive probabilities.

When $\mathcal{C}$ is a field we say that $P$ satisfies countable additivity (\cite{Billingsley}[p.~20]) if, for every sequence $(A_{i})_{i \geq 1}$ of disjoint sets of $\mathcal{C}$ such that $\cup_{i \geq 1}{A_{i}} \in \mathcal{C}$, $P(\cup_{i \geq 1}{A_{i}}) = \sum_{i=1}^{\infty}{P(A_{i})}$. If an assignment of prices is a finitely additive probability and also satisfies countable additivity, we call it a countably additive probability.

We address two questions related to de Finetti's \cite{deFinetti2} results. First, why can the broker perform only a finite number of bets? For example, consider bets on successive flips of a coin. A strategy such as ``bet on heads until heads is observed'' yields a countable number of bets because there is no finite upper bound on the number of heads required.  Is a countable number of bets related to countably additive probabilities in the same way a finite number of bets is related to finitely additive probabilities? Second, is the assumption that $\mathcal{C}$ is a field necessary? \cite{Heath} answers this question when only finitely many bets are considered. If a given assignment of prices on an arbitrary collection $\mathcal{C} \subset \mathcal{P}(\Omega)$ is coherent in de Finetti's \cite{deFinetti2} setting then these prices can be extended to a finitely additive probability on the smallest field which contains $\mathcal{C}$, $\mathcal{F}(\mathcal{C})$. That is, there exists $P^{e}: \mathcal{F}(\mathcal{C}) \rightarrow [0,1]$ such that, if $C \in \mathcal{C}$, $P^{e}(C) = P(C)$ and $P^{e}$ is a finitely additive probability. Thus, using \cite{deFinetti2}, an assignment of prices is coherent on $\mathcal{C}$ if and only if $P$ can be extended to a coherent assignment on $\mathcal{F}(\mathcal{C})$. Do similar results hold when one considers a countable number of bets?

Before these questions can be addressed, it is necessary to define what a ``countable number of bets'' means. For example, consider an event $A \in \mathcal{C}$. Take a sequence of bets in which a unit of event $A$ is bought for every odd number and sold for every even number. It's not clear how to define the price of this bet or what would result if $A$ were observed.

In order to define which bets are considered, we first introduce notation. Let $\mathbb{Z}^{+} = \{1,2,3,\ldots\}$ and $\Omega$ be the sample space. You assign prices to events which are subsets of $\Omega$. The set of all events for which prices are assigned is denoted by $\mathcal{C}$. A price assignment is a function $P: \mathcal{C} \rightarrow \mathbb{R}$. For every $A \in \mathcal{C}$, $P(A)$ is the price you assign to $A$. $\mathcal{F}(\mathcal{C})$ and $\mathcal{\sigma}(\mathcal{C})$ are, respectivelly, the field and the $\sigma$-field generated by $\mathcal{C}$. 

If $\mathcal{C}$ is a field and $P$ is a countably additive probability, Carath\'{e}odory's extension theorem (\cite{Billingsley}[p.~32]) guarantees that there exists a unique countably additive extension of $P$ to $\sigma(\mathcal{C})$. Denote this extension by $P^{*}$. Let $\mathcal{U} = \{(a,b): a,b \in \mathbb{R}\}$. A random variable $X$ is $\sigma(\mathcal{C})$-measurable if, for every $U \in \sigma(\mathcal{U})$, $X^{-1}[U] \in \sigma(\mathcal{C})$. For $\sigma(\mathcal{C})$-measurable $X$ and $Y$, $X$ is a version of $Y$ if $X=Y$ a.s. $P^{*}$.

A betting portfolio is a sequence of bets $(\alpha_{i},A_{i})_{i \in \mathbb{Z}^{+}}$, in which the broker buys $\alpha_{i} \in \mathbb{R}$ units of event $A_{i} \in \mathcal{C}$ from you (if $\alpha_{i}$ is negative, the broker sells $\alpha_{i}$ units of $A_{i}$). The price of the betting portfolio is $\sum_{i=1}^{\infty}{\alpha_{i}P(A_{i})}$. For every $A \subset \Omega$, let $I_{A}:\Omega \rightarrow \{0,1\}$ be the indicator function of $A$, that is, $I_{A}(w) = 1$, if $w \in A$, and $I_{A}(w) = 0$, otherwise. The balance of the betting portfolio is $\sum_{i=1}^{\infty}{\alpha_{i}(I_{A_{i}}-P(A_{i}))}$. A betting system is a collection of betting portfolios.

\begin{definition}
  \label{coherence}
  $P$ is incoherent in a given betting system if there exists a betting portfolio $(\alpha_{i},A_{i})_{i \in \mathbb{Z}^{+}}$ in that system which leads to uniform sure loss. That is, there exists $\epsilon > 0$ such that, $\forall w \in \Omega$, $\sum_{i=1}^{\infty}{\alpha_{i}(I_{A_{i}}(w)-P(A_{i}))} \leq -\epsilon$. $P$ is coherent if it is not incoherent.
\end{definition}

Definition \ref{coherence} corresponds to the one presented in \cite{deFinetti2}. \cite{Adams} observes that the coherence of $P$ depends on what types of convergence are specified by a betting system. Adams considers a betting system which includes exactly those betting portfolios $(\alpha_{i},A_{i})_{i \in \mathbb{Z}^{+}}$ such that $\sum_{i=1}^{\infty}{|\alpha_{i}(I_{A_{i}}-P(A_{i}))|}$ converges for all $w \in \Omega$. In this betting system, $P$ being a countably additive probability is a necessary condition for $P$ to be coherent. Is it sufficient?
 
\cite{Beam} establishes a partial answer. Assume that $\mathcal{C}$ is a field. Consider a betting system which includes exactly those betting portfolios $(\alpha_{i},A_{i})_{i \in \mathbb{Z}^{+}}$ such that the random quantity $\sum_{i=1}^{\infty}{|\alpha_{i}(I_{A_{i}}-P(A_{i}))|}$ is bounded on $\Omega$. Beam \cite{Beam} proves that $P$ is coherent in such a betting system iff $P$ is a countably additive probability on $\mathcal{C}$. 

Nevertheless, Beam \cite{Beam} argues that the condition imposed on the betting system might be artificial. So he weakens this condition and includes every betting portfolio such that $\sum_{i=1}^{\infty}{\alpha_{i}(I_{A_{i}}-P(A_{i}))}$ converges pointwise and is bounded on $\Omega$. In this betting system, Beam \cite{Beam} proves that if $\Omega=(0,1)$ and $\mathcal{C}$ is the Borel field on $(0,1)$, then the assignment of prices which corresponds to the uniform distribution is incoherent. Beam \cite{Beam} constructs a betting portfolio such that the price is conditionally convergent and, through a permutation of the indexes of the portfolio, obtains uniform sure loss.

We consider two additional betting systems besides the ones in \cite{deFinetti2}, \cite{Adams} and \cite{Beam}. 

\begin{definition}
  \label{bettingSystems}
  Let $(\Omega,\mathcal{C},P)$ be given; the following betting systems include exactly those betting portfolios $(\alpha_{i},A_{i})_{i \in \mathbb{Z}^{+}}$ such that:

  \textbf{System 1}: Only finitely many $\alpha_{i} \neq 0$ (\cite{deFinetti2}).
	
  \textbf{System 2}: $\sum_{i=1}^{\infty}{|\alpha_{i}|P(A_{i})} < \infty$ and $\sum_{i=1}^{\infty}{\alpha_{i}(I_{A_{i}}-P(A_{i}))}$ converges pointwise.
	
  \textbf{System 2B}: $\sum_{i=1}^{\infty}{|\alpha_{i}(I_{A_{i}}-P(A_{i}))|}$ is bounded on $\Omega$ (\cite{Beam}). 
	
  \textbf{System 2A}: For all $w \in \Omega$, $\sum_{i=1}^{\infty}{|\alpha_{i}(I_{A_{i}}-P(A_{i}))|} < \infty$ (\cite{Adams}).
	
  \textbf{System 3}: $\sum_{i=1}^{\infty}{\alpha_{i}P(A_{i})}$ converges and $\sum_{i=1}^{\infty}{\alpha_{i}(I_{A_{i}}-P(A_{i}))}$ converges pointwise.
\end{definition}

Betting systems $2$, $2B$, $2A$ and $3$ are extensions of system $1$ and include the possibility of a countable number of bets. Betting system $2$ is the largest betting system such that the price of every betting portfolio does not depend on the order in which the bets are settled and the balance is defined for every $w \in \Omega$. System $2A$ extends system $2B$. System $2A$ does not extend system $2$ nor does system $2$ extend system $2B$, as the following examples show. Let $P$ be the uniform distribution on $[0,1]$. The portfolio $(\alpha_{i},A_{i})_{i\in \mathbb{Z}^{+}}$ such that $\alpha_{i} =\frac{(-1)^{i}}{i}$ and $A_{i}=[0,\frac{1}{i^{2}})$ is included in system $2$ but not in $2A$. Also, the portfolio $(\alpha_{i},A_{i})_{i\in \mathbb{Z}^{+}}$ such that $\alpha_{i}=1$ and $A_{i}=[0,1]$ for every $i \in \mathbb{Z}^{+}$ is included in system $2B$ but not in $2$. Betting system $3$ extends $2$ and is the largest betting system such that, for every included betting portfolio, price and balance are 
defined. 

When $\mathcal{C}$ is a field, for each of these systems we present conditions on $P$ that are equivalent to coherence. In Section \ref{altCoherence} we show that the same conditions on $P$ are also equivalent to alternatives to the definition of coherence (\cite{Adams}). In particular, we solve a question in \cite{Adams} about what condition on $P$ is equivalent to system-$2A$-rationality (defined in Section \ref{altCoherence}). We also provide counter-examples to extension results such as in \cite{Heath} in all betting systems besides system $1$. The existence of such an extension in system $2B$ was originally questioned in \cite{Beam}. Finally we characterize what balances can be generated by a coherent $P$ in each of the betting systems.

The following definition and lemma will be useful in many of the proofs:

\begin{definition}
  \label{defConstantBet}
  A betting portfolio $(\alpha_{i},A_{i})_{i \in \mathbb{Z}^{+}}$ is constant if $\sum_{i=1}^{\infty}{\alpha_{i}(I_{A_{i}}-P(A_{i}))}$ is constant on $\Omega$.
\end{definition}

\begin{lemma}
  \label{constantBet}
  In betting systems $1$, $2$, $2A$, $2B$ and $3$, if a constant betting portfolio with a balance different from $0$ is included in the system, then $P$ is incoherent in that system.
\end{lemma}

\begin{proof}
  Let $(\alpha_{i},A_{i})_{i \in \mathbb{Z}^{+}}$ be a constant betting portfolio. If its balance is negative, then $P$ is incoherent. If its balance is positive, then $(-\alpha_{i},A_{i})_{i \in \mathbb{Z}^{+}}$ is in the betting system, is constant and has a negative balance. Hence, $P$ is incoherent. 
\end{proof}

\section{Betting system 2}
\label{bet2}

Next, Theorem \ref{bet2equiv} establishes a necessary and sufficient condition for $P$ to be system-2-coherent when $\mathcal{C}$ is a field.

\begin{theorem}
  \label{bet2equiv}
  If $\mathcal{C}$ is a field, $P$ is system-2-coherent if and only if $P$ is a countably additive probability.
\end{theorem}

\begin{proof}
  Lemma \ref{bet2countadd} proves that every system-2-coherent $P$ is countably additive. Lemma \ref{bet2coh} proves the reverse implication.
\end{proof} 

\begin{lemma}
  \label{bet2countadd}
  Let $\mathcal{C}$ be a field. If $P$ is system-$2$-coherent, then $P$ is a countably additive probability.
\end{lemma}

\begin{proof}
  Let $P$ be system-2-coherent and $(A_{i})_{i \geq 2}$, be a sequence in $\mathcal{C}$ such that, for all $i \neq j$, $A_{i} \cap A_{j} = \emptyset$ and $A_{1} = \cup_{i=2}^{\infty}{A_{i}} \in \mathcal{C}$. 
	
  Consider a betting portfolio $(\beta_{i},B_{i})_{i \in Z^{+}}$ such that $(\beta_{1},B_{1}) = (1,A_{1})$ and, for $i \geq 2$, $(\beta_{i}, B_{i}) = (-1,A_{i})$. The price of this portfolio is $P(A_{1}) - \sum_{i=2}^{\infty}{P(A_{i})}$. 
	
  The following argument proves that $(\beta_{i},B_{i})_{i \in Z^{+}}$ is in betting system $2$. Since $P$ is system-2-coherent, it is also system-1-coherent. Thus, using \cite{Heath}, there exists a finitely additive extension of $P$ to $\mathcal{F}(\mathcal{C})$. Hence, for all $i \in \mathbb{Z}^{+}$, $P(A_{i}) \geq 0$ and for all $n \in \mathbb{Z}^{+}$, $\sum_{i=2}^{n}{P(A_{i})} \leq 1$. Conclude that $P(A_{1}) - \sum_{i=2}^{\infty}{P(A_{i})}$ converges absolutely. Also, since $A_{i}$ are disjoint, $\{w \in \Omega: I_{B_{i}}(w) = 1\text{ infinitely often}\} = \emptyset$ and $\sum_{i=1}^{\infty}{\beta_{i}(I_{B_{i}}-P(B_{i}))}$ converges pointwise on $\Omega$.
	
  Since this betting portfolio is constant and $P$ is system-$2$-coherent, Lemma \ref{constantBet} implies that $P(A_{1}) = \sum_{i=2}^{\infty}{P(A_{i})}$.
\end{proof}

\begin{lemma}
  \label{bet2coh}
  Let $\mathcal{C}$ be a field and $P$ be a countably additive probability. For every betting portfolio $(\alpha_{i},A_{i})_{i \in \mathbb{Z}^{+}}$ in betting system $2$, $E_{P^{*}}[\sum_{i=1}^{\infty}{\alpha_{i}(I_{A_{I}}-P(A_{i}))}] = 0$. Hence, $P$ is system-$2$-coherent.
\end{lemma}

\begin{proof}
  Let $P$ be a countably additive probability and $(\alpha_{i},A_{i})_{i \in \mathbb{Z}^{+}}$ be an arbitrary betting portfolio in betting system $2$. Let $(\beta_{i},B_{i})_{i \in \mathbb{Z}^{+}}$ be the subsequence of $(\alpha_{i},A_{i})_{i \in \mathbb{Z}^{+}}$ such that $\alpha_{i} \geq 0$ and $(\gamma_{i},C_{i})_{i \in \mathbb{Z}^{+}}$ be the subsequence such that $\alpha_{i} < 0$. It follows from the Monotone Convergence Theorem (\cite{Billingsley}[p.~211]) that
  
\begin{align}
  \label{bet2cohEqn1}
    E_{P^{*}}\left[\sum_{i=1}^{\infty}{\beta_{i}I_{B_{i}}}\right]   &= \sum_{i=1}^{\infty}{\beta_{i}P(B_{i})} \leq \sum_{i=1}^{\infty}{|\alpha_{i}|P(A_{i})} < \infty \\
  \label{bet2cohEqn2}
    E_{P^{*}}\left[\sum_{i=1}^{\infty}{-\gamma_{i}I_{C_{i}}}\right] &= \sum_{i=1}^{\infty}{-\gamma_{i}P(C_{i})} \leq \sum_{i=1}^{\infty}{|\alpha_{i}|P(A_{i})} < \infty
\end{align}

Hence, $\sum_{i=1}^{\infty}{\beta_{i}I_{B_{i}}}$ and $\sum_{i=1}^{\infty}{\gamma_{i}I_{C_{i}}}$ converge a.s. $P^{*}$. Let $X_{B} = \sum_{i=1}^{\infty}{\beta_{i}I_{B_{i}}}$ and $X_{C} = \sum_{i=1}^{\infty}{\gamma_{i}I_{C_{i}}}$. It follows from equations \eqref{bet2cohEqn1} and \eqref{bet2cohEqn2} that

\begin{align*}
  E_{P^{*}}\left[\sum_{i=1}^{\infty}{\alpha_{i}(I_{A_{i}}-P(A_{i}))}\right] = E_{P^{*}}\left[X_{B}+X_{C}-E_{P^{*}}[X_{B}]-E_{P^{*}}[X_{C}]\right] = 0
\end{align*}

Thus, there exists $w \in \Omega$, such that $\sum_{i=1}^{\infty}{\alpha_{i}(I_{A_{i}}(w)-P(A_{i}))} \geq 0$ and the portfolio doesn't lead to sure loss. Conclude that $P$ is system-2-coherent. \end{proof}

Next, we show an example such that $P$ is coherent in betting system $2$ and admits no countably additive extension to $\mathcal{F}(\mathcal{C})$. That is, in betting system $2$, a coherent assignment of prices on $\mathcal{C}$ might not admit a coherent extension to $\mathcal{F}(\mathcal{C})$. In this respect, betting system $2$ differs from system $1$.

\begin{example}
  \label{bet2Example}
  Consider that $\mathcal{C}$ is a countable set of events which can be ordered in a sequence $(A_{i})_{i \geq 1}$ such that $A_{i} \searrow \emptyset$. Also consider that $1 \geq P(A_{i}) \searrow \delta > 0$. For example, take $\Omega = (0,1)$, $A_{i} = (0,\frac{1}{i+1})$ and $P((0,\frac{1}{i+1})) = \frac{1}{i+1} + \delta$. $P$ is system-2-coherent but admits no coherent extension to $\mathcal{F}(\mathcal{C})$. 
\end{example}

\begin{proof}
  Assume $P$ admits a system-2-coherent extension to $\mathcal{F}(\mathcal{C})$. Hence, by Theorem \ref{bet2equiv} and Carath\'{e}odory's extension theorem, there exists an extension of these prices to $\sigma(\mathcal{C})$ which is a countably additive probability. This is a contradiction since $A_{i} \searrow \emptyset$ and $P(A_{i}) \searrow \delta > 0$. Hence, $P$ doesn't admit a system-$2$-coherent extension to $\mathcal{F}(\mathcal{C})$.

  Next, we prove by contradiction that $P$ is system-2-coherent on $\mathcal{C}$. Assume there exists $(\beta_{i},B_{i})_{i \in \mathbb{Z}^{+}}$ in betting system $2$ which leads to uniform sure loss. From the definition of betting system $2$, $\sum_{i=1}^{\infty}{|\beta_{i}|P(B_{i})} < \infty$ and thus:

  \begin{align}
    \label{bet2ExampleEqn1}
    \sum_{i=1}^{\infty}{|\beta_{i}(I_{B_{i}}-P(B_{i}))|}	&\leq \sum_{i=1}^{\infty}{|\beta_{i}|I_{B_{i}}} + \sum_{i=1}^{\infty}{|\beta_{i}|P(B_{i})}\\
								&\leq (1+\frac{1}{\delta})\sum_{i=1}^{\infty}{|\beta_{i}|P(B_{i})} < \infty \notag
  \end{align}

  Let $\mu$ be the counting measure on $\mathbb{Z}^{+}$ and $f: \mathbb{Z}^{+} \rightarrow \mathbb{Z}^{+}$, where $f(i)$ is the integer such that $B_{i} = A_{f(i)}$.

  \begin{align}
    \label{bet2ExampleEqn2}
    \sum_{i=1}^{\infty}{\beta_{i}(I_{B_{i}}-P(B_{i}))} = \int_{\mathbb{Z}^{+}}{\int_{\mathbb{Z}^{+}}{I_{\{f(i)\}}(j)\beta_{i}(I_{A_{j}}-P(A_{j}))\mu(dj)}\mu(di)}
  \end{align}

  Hence, from equations \eqref{bet2ExampleEqn1} and \eqref{bet2ExampleEqn2}, for all $w \in \Omega$,

  \begin{align}
    \label{bet2ExampleEqn3}
    \int_{\mathbb{Z}^{+}}{\int_{\mathbb{Z}^{+}}{I_{\{f(i)\}}|\beta_{i}(I_{A_{j}}(w)-P(A_{j}))|\mu(dj)}\mu(di)} < \infty
  \end{align}

  Conclude from applying Fubini's Theorem (\cite{Billingsley}[p.~238]) to equations \eqref{bet2ExampleEqn2} and \eqref{bet2ExampleEqn3} that

  \begin{align}
    \label{bet2ExampleEqn4}
    \sum_{i=1}^{\infty}{\beta_{i}(I_{B_{i}}-P(B_{i}))}	=	\sum_{j=1}^{\infty}{\left(\sum_{i \in f^{-1}[j]}{\beta_{i}}\right)(I_{A_{j}}-P(A_{j}))}
  \end{align}

  Define $\alpha_{i} = \left(\sum_{j \in f^{-1}[i]}{\beta_{j}}\right)$. Equation \eqref{bet2ExampleEqn4} and the sure loss of $(\beta_{i},B_{i})_{i \in \mathbb{Z}^{+}}$ imply that $(\alpha_{i},A_{i})_{i \in \mathbb{Z}^{+}}$ leads to sure loss, which equivalently means,

  \begin{align*}
    \exists \epsilon > 0: \forall w \in \Omega, \sum_{i=1}^{\infty}{\alpha_{i}I_{A_{i}}} \leq \sum_{i=1}^{\infty}{\alpha_{i}P(A_{i})} - \epsilon
  \end{align*}

  By construction, $\sum_{i=1}^{\infty}{|\alpha_{i}|P(A_{i})} \leq \sum_{i=1}^{\infty}{|\beta_{i}|P(B_{i})} < \infty$. Hence,

  \begin{align}
    \label{bet2ExampleEqn5}
    \exists \epsilon^{*} > 0, \exists k^{*} \in \mathbb{Z}^{+}: \forall w \in \Omega, \sum_{i=1}^{\infty}{\alpha_{i}I_{A_{i}}} \leq \sum_{i=1}^{k^{*}}{\alpha_{i}P(A_{i})} - \epsilon^{*}
  \end{align}

  Since $A_{i} \searrow \emptyset$, for all $w \in \Omega-A_{k^{*}+1}$, $\sum_{i=1}^{\infty}{\alpha_{i}I_{A_{i}}(w)} = \sum_{i=1}^{k^{*}}{\alpha_{i}I_{A_{i}}(w)}$. Let $w^{*} \in A_{k^{*}}-A_{k^{*}+1}$. Similarly, $\forall w \in A_{k^{*}+1}$, $\sum_{i=1}^{\infty}{\alpha_{i}I_{A_{i}}(w)} = \sum_{i=1}^{k^{*}}{\alpha_{i}I_{A_{i}}(w^{*})}$. It follows from equation \eqref{bet2ExampleEqn5} that

  \begin{align*}
    \forall w \in \Omega, \sum_{i=1}^{k^{*}}{\alpha_{i}(I_{A_{i}}-P(A_{i}))} \leq -\epsilon^{*}
  \end{align*}

  Thus, $P$ is system-1-incoherent. Since $P$ can be extended to a finitely additive probability on $\mathcal{F}(\mathcal{C})$, it follows from \cite{deFinetti2} that $P$ is system-$1$-coherent, a contradiction. Conclude that $P$ is system-2-coherent.
\end{proof}

We now consider that $P$ is a fixed system-2-coherent assignment on a field $\mathcal{C}$ and inquire about the space of balances generated by betting portfolios in betting system $2$. Theorem \ref{bet2Balances} characterizes this space as all $\sigma(\mathcal{C})$-measurable random variables, $X$, such that $E_{P^{*}}[X] = 0$. Observe that, by construction of betting system $2B$ all balances, $X$, are such that $E_{P^{*}}[X] = 0$ and $X$ is bounded. Hence, the space of balances generated by betting system $2$ is larger than that generated by $2B$.

\begin{theorem}
  \label{bet2Balances}
  Let $\mathcal{C}$ be a field and $P$ be system-$2$-coherent. For every $\sigma(\mathcal{C})$-measurable random variable $X$, there exists a betting portfolio in betting system 2 with balance $B$ which is a version of $X$ iff $E_{P^{*}}[X] = 0$.
\end{theorem}

\begin{proof}
  Let $(\alpha_{i},A_{i})_{i \in \mathbb{Z}^{+}}$ be a betting portfolio in betting system $2$. Lemma \ref{bet2coh} implies that $E_{P^{*}}[\sum_{i=1}^{\infty}{\alpha_{i}(I_{A_{i}}-P(A_{i}))}] = 0$. Hence, if $\sum_{i=1}^{\infty}{\alpha_{i}(I_{A_{i}}-P(A_{i}))} = X$ a.s. $P^{*}$, $E_{P^{*}}[X] = 0$.

  Assume $E_{P^{*}}[X] = 0$. Let $X^{+} = X \cdot I_{X > 0}$ and $X^{-} = X \cdot I_{X < 0}$. Using \cite{Billingsley}[p.~186], there exists a sequence of $\sigma(\mathcal{C})$-measurable non-negative simple functions, $X^{+}_{n}$, such that $X^{+}_{n} \nearrow X^{+}$ pointwise on $\Omega$.
	
  Take $Y_{1}=X^{+}_{1}$ and, for $n \geq 2$, $Y_{n} = X^{+}_{n}-X^{+}_{n-1}$. $Y_{n}$ are $\sigma(\mathcal{C})$-measurable non-negative simple functions such that $\sum_{i=1}^{n}{Y_{i}} \nearrow X^{+}$. For all $i \in \mathbb{Z}^{+}$, there exists $m_{i} \in \mathbb{Z}^{+}$, $\sigma(\mathcal{C})$-measurable $B_{i,1},\ldots,B_{i,m_{i}}$ and positive $\beta_{i,1},\ldots,\beta_{i,m_{i}}$ such that $Y_{i} = \sum_{j=1}^{m_{i}}{\beta_{i,j}I_{B_{i,j}}}$. Hence, $X^{+} = \sum_{i=1}^{\infty}{\sum_{j=1}^{m_{i}}{\beta_{i,j}I_{B_{i,j}}}}$.
	
  Using \cite{Billingsley}[p.~167], for each $B_{i,j}$, there exists a sequence $(B_{i,j,k})_{k \in \mathbb{Z}^{+}}$ of disjoint events on $\mathcal{C}$ such that, for all $k$, $B_{i,j,k} \subset B_{i,j}$ and $P^{*}(B_{i,j}-\bigcup_{k \in \mathbb{Z}^{+}}{B_{i,j,k}}) = 0$. Hence, conclude that $X^{+} = \sum_{i=1}^{\infty}{\sum_{j=1}^{m_{i}}{\sum_{k=1}^{\infty}{\beta_{i,j}I_{B_{i,j,k}}}}}$ a.s. $P^{*}$. 
	
  Take $f$ as a bijection from $\mathbb{Z}^{+}$ to the span of the indexes $(i,j,k)$ and let $\beta_{i,j,k} = \beta_{i,j}$. $X^{+} = \sum_{i=1}^{\infty}{\beta_{f(i)}I_{B_{f(i)}}}$ a.s. $P^{*}$. It follows from the Monotone Convergence Theorem that $\sum_{i=1}^{\infty}{\beta_{f(i)}P(B_{f(i)})} = E_{P^{*}}[X^{+}]$. Also, for every $w^{*} \in \{w \in \Omega: \sum_{i=1}^{\infty}{\beta_{f(i)}I_{B_{f(i)}}} \neq X^{+}\}$, it follows by construction that $\sum_{i=1}^{\infty}{\beta_{f(i)}I_{B_{f(i)}}(w^{*})} = 0$.
	
  Consider the same construction such that $-\sum_{i=1}^{\infty}{\gamma_{g(i)}I_{C_{g(i)}}} = X^{-}$ a.s. $P^{*}$ and also that $-\sum_{i=1}^{\infty}{\gamma_{g(i)}P(C_{g(i)})} = E_{P^{*}}[X^{-}]$. Let $(\alpha_{i},A_{i})_{i \in \mathbb{Z}^{+}}$ be the betting portfolio such that $(\alpha_{2i},A_{2i})=(\beta_{i},B_{i})$ and $(\alpha_{2i-1},A_{2i-1})=(-\gamma_{i},C_{i})$. Thus,	

  \begin{align*}
    \sum_{i=1}^{\infty}{\alpha_{i}(I_{A_{i}}-P(A_{i}))} =  \sum_{i=1}^{\infty}{\alpha_{i}I_{A_{i}}} - E[X] = \sum_{i=1}^{\infty}{\alpha_{i}I_{A_{i}}}
  \end{align*}

  Conclude that $\sum_{i=1}^{\infty}{\alpha_{i}(I_{A_{i}}-P(A_{i}))} = X$ a.s. $P^{*}$. It remains to show that $(\alpha_{i},A_{i})_{i \in \mathbb{Z}^{+}}$ is in betting system $2$. First, $\sum_{i=1}^{\infty}{|\alpha_{i}|P(A_{i})} = E_{P^{*}}[X^{+}]-E_{P^{*}}[X^{-}] < \infty$. Next, for any $w^{*} \in \{w \in \Omega: \sum_{i=1}^{\infty}{\alpha_{i}I_{A_{i}}} \neq X\}$, $\sum_{i=1}^{\infty}{\alpha_{i}I_{A_{i}}(w^{*})} = 0$. Conclude that $\sum_{i=1}^{\infty}{\alpha_{i}(I_{A_{i}}-P(A_{i}))}$ converges pointwise on $\Omega$ and $(\alpha_{i},A_{i})_{i \in \mathbb{Z}^{+}}$ is in betting system 2.
\end{proof}

\section{Betting system 3}
\label{bet3}

Next, Theorem \ref{bet3equiv} establishes a necessary and sufficient condition for $P$ to be system-3-coherent when $\mathcal{C}$ is a field. We call $A$ an atom if $P(A) > 0$ and $\forall B \in \mathcal{C}$ such that $B \subset A$, $P(B) = P(A)$ or $P(B) = 0$.

\begin{theorem}
  \label{bet3equiv}
  If $\mathcal{C}$ is a field, $P$ is system-3-coherent iff $P$ is a countably additive probability and $\Omega$ admits a finite partition in $\mathcal{C}$-measurable atoms.
\end{theorem}

\begin{proof}
  We prove Theorem \ref{bet3equiv} in three steps. First, we establish Lemma \ref{finiteSupport}, which states two conditions equivalent to being able to partition $\Omega$ into a finite number of atoms. Lemma \ref{bet3finiteSupport} establishes that a price assignment is coherent in betting system $3$ only if it is a countably additive probability and $\Omega$ admits a finite partition in $\mathcal{C}$-measurable atoms. Lemma \ref{bet3coh} establishes the reverse implication.
\end{proof}

\begin{definition}
  \label{p-finite}
  Let $P$ be a countably additive probability. P is p-finite \cite{Adams} if, for every $(B_{i})_{i \in \mathbb{Z}^{+}}$ in $\mathcal{F}$ such that, for all $i \in \mathbb{Z}^{+}$, $B_{i+1} \subset B_{i}$ and $\lim_{i \rightarrow \infty}{P(B_{i})} = 0$, there exists $j \in \mathbb{Z}^{+}$ such that $P(B_{j}) = 0$.
\end{definition}

\begin{lemma}
  \label{finiteSupport}
  Consider $(\Omega, \mathcal{F}, P)$ such that $\mathcal{F}$ is a field on $\Omega$ and $P$ is a countably additive probability. The following propositions are equivalent:

  \begin{enumerate}[label=(\roman{*}), ref=(\roman{*})]
    \item There is no sequence $(A_{i})_{i \in \mathbb{Z}^{+}}$ of mutually exclusive events of $\mathcal{F}$ such that $P(A_{i}) > 0$ for all $i \in \mathbb{Z}^{+}$.
    \item $\Omega$ admits a finite partition in measurable atoms.
    \item $P$ is p-finite.
  \end{enumerate}
\end{lemma}

\begin{proof}
  \textbf{$(i \rightarrow ii)$}: A measurable set $F$ is finitely atomizable (FA) if $F$ admits a partition in a finite number of atoms of $\mathcal{C}$.

  We prove $(\text{not } ii \rightarrow \text{not } i)$ by constructing $(A_{i})_{i \in \mathbb{Z}^{+}}$ such that for $i \neq j$, $A_{i} \cap A_{j} = \emptyset$ and $P(A_{i}) > 0$. In the following passage, we construct $(A_{i})_{i \in \mathbb{Z}^{+}}$ together with an auxiliary sequence $(B_{i})_{i \in \mathbb{Z}^{+}}$ such that $B_{i}$ is not FA and $B_{i} \cap A_{j} = \emptyset$ for $j \leq i$. $A_{i}$ and $B_{i}$ are defined in terms of $B_{i-1}$.

  Choose $B_{0} = \Omega$. Next, we show how to construct $B_{i}$ and $A_{i}$ in terms of $B_{i-1}$. By assumption, $B_{i-1}$ is not FA. Hence $B_{i-1}$ is not an atom and there exist measurable $F_{0,i}$ and $F_{1,i}$ which partition $B_{i-1}$ such that $P(F_{0,i}) > 0$ and $P(F_{1,i}) > 0$. Since $B_{i-1}$ is not FA, there exists $k \in \{0,1\}$ such that $F_{k,i}$ is not FA. Since $B_{i-1} \cap A_{j} = \emptyset$ for $j \leq i-1$, defining $A_{i} = F_{1-k,i} \subset B_{i-1}$, $(A_{j})_{j \leq i}$ are mutually disjoint. Taking $B_{i} = F_{k,i}$, $B_{i}$ is not FA and $B_{i} \cap A_{j} = \emptyset$ for $j \leq i$. By construction, $(A_{i})_{i \in \mathbb{Z}^{+}}$ is a sequence of mutually disjoint events with positive probability.

  \

  \textbf{$(ii \rightarrow iii)$}: Let $R = \{F_{1},\ldots,F_{n}\}$ be a partition of $\Omega$ in $\mathcal{F}$-measurable atoms. Since $P$ is finitely additive, $\forall i \in \mathbb{Z}^{+}$, $P(B_{i}) = \sum_{j=1}^{n}{P(B_{i} \cap F_{j})}$. Hence, since $F_{i}$ are atoms, if $P(B_{i}) \neq 0$, $P(B_{i}) \geq \min_{j \in \{1,\ldots,n\}}{P(F_{j})} > 0$. Thus, if $\lim_{i \rightarrow \infty}{P(B_{i})} = 0$, there exists $j \in \mathbb{Z}^{+}$ such that $P(B_{j}) = 0$.

  \

  \textbf{$(iii \rightarrow i)$}: We prove $(\text{not } i \rightarrow \text{not } iii)$. Let $(A_{i})_{i \in \mathbb{Z}^{+}}$ be disjoint events such that $\forall i \in \mathbb{Z}^{+}, P(A_{i}) > 0$. Let $B_{j} = \bigcup_{i=j}^{\infty}{A_{i}}$. By construction, $B_{j} \searrow \emptyset$ and $P(B_{j}) > 0$ for every $j \in \mathbb{Z}^{+}$. Since, $B_{j} \searrow \emptyset$ and $P$ is a countably additive probability, $\lim_{j \rightarrow \infty}{P(B_{j})} = 0$.
\end{proof}

\begin{lemma}
  \label{bet3finiteSupport} 
  If $\mathcal{C}$ is a field and $P$ is system-$3$-coherent, then $P$ is a countably additive probability such that $\Omega$ can be partitioned in a finite number of $\mathcal{C}$-measurable atoms.
\end{lemma}

\begin{proof}
  Since every betting portfolio in betting system $3$ is also in betting system $2$, Lemma \ref{bet2countadd} implies that every system-3-coherent is a countably additive probability.

  Using Lemma \ref{finiteSupport} it remains to show that, if $P$ is system-3-coherent, there is no sequence of events $(A_{i})_{i \in \mathbb{Z}^{+}}$ such that, for all $i \neq j$, $A_{i} \cap A_{j} = \emptyset$ and for all $i \in \mathbb{Z}^{+}$, $P(A_{i}) > 0$. Assume there exists such a sequence. Consider the betting portfolio $(\frac{(-1)^{i}}{iP(A_{i})},A_{i})_{i \in \mathbb{Z}^{+}}$. The price of this portfolio is $p = \sum_{i=1}^{\infty}{\frac{(-1)^{i}}{i}}$. Since the price converges conditionally, there exists a permutation of $\mathbb{Z}^{+}$ - $\pi$ - such that $p^{*} = \sum_{i=1}^{\infty}{\frac{(-1)^{\pi(i)}}{\pi(i)}} \neq p$.

  Consider the betting portfolio $(\beta_{i},B_{i})_{i \in \mathbb{Z}^{+}}$ such that $(\beta_{2i-1},B_{2i-1}) = (\frac{(-1)^{i}}{iP(A_{i})},A_{i})$ and $(\beta_{2i},B_{2i}) = (\frac{(-1)^{\pi(i)+1}}{\pi(i)P(A_{\pi(i)})},A_{\pi(i)})$. The balance of this portfolio is 

  \begin{align*}
    \sum_{i=1}^{\infty}{\beta_{i}(I_{B_{i}}-P(B_{i}))}	&= \sum_{i}^{\infty}{\frac{(-1)^{i}}{iP(A_{i})}\left(I_{A_{i}}-P(A_{i})\right)} + \sum_{i}^{\infty}{\frac{(-1)^{\pi(i)+1}}{\pi(i)P(A_{\pi(i)})}\left(I_{A_{\pi(i)}}-P(A_{\pi(i)})\right)} =\\																											&= p-p^{*}+\sum_{i=1}^{\infty}{\left(\frac{(-1)^{i}I_{A_{i}}}{iP(A_{i})} - \frac{(-1)^{\pi(i)}I_{A_{\pi(i)}}}{\pi(i)P(A_{\pi(i)})} \right) = p-p^{*} \neq 0} 
  \end{align*}
\end{proof}

The last equality follows since the $A_{i}$ are disjoint. Since $(\beta_{i},B_{i})_{i \in \mathbb{Z}^{+}}$ is included in system $3$ and has a constant balance different from $0$, conclude from Lemma \ref{constantBet} that $P$ is system-3-incoherent, a contradiction.

\begin{lemma}
  \label{bet3coh}
  If $\mathcal{C}$ is a field, $P$ is a countably additive probability and $\Omega$ admits a finite partition in $\mathcal{C}$-measurable atoms, then for any betting portfolio $(\alpha_{i},A_{i})_{i \in \mathbb{Z}^{+}}$ in betting system $3$, $E_{P^{*}}[\sum_{i=1}^{\infty}{\alpha_{i}(I_{A_{i}}-P(A_{i}))}] = 0$. Thus, $P$ is system-$3$-coherent.
\end{lemma}

\begin{proof}
  Let $\{F_{1},\ldots,F_{n}\}$ be a partition of $\Omega$ into atoms. Observe that, if all expectations are defined, then
  \begin{align}
    \label{bet3cohEqn1}
    E_{P^{*}}\left[\sum_{i=1}^{\infty}{\alpha_{i}(I_{A_{i}}-P(A_{i}))}\right] &= E_{P^{*}}\left[\sum_{i=1}^{\infty}{\alpha_{i}\left(\sum_{j=1}^{n}{I_{A_{i} \cap F_{j}}-P(A_{i} \cap F_{j})}\right)}\right] = \nonumber \\
										&= \sum_{j=1}^{n}{E_{P^{*}}\left[\sum_{i=1}^{\infty}{{\alpha_{i}(I_{A_{i} \cap F_{j}}-P(A_{i} \cap F_{j}))}}\right]}
  \end{align}

  Let $I_{j} = \{i \in \mathbb{Z}^{+}: P(A_{i} \cap F_{j}) = P(F_{j})\}$. For every $j \in \{1,\ldots,n\}$, since $\{F_{1},\ldots,F_{n}\}$ is a partition of $\Omega$ into atoms,
  \begin{align}
    \label{bet3cohEqn2}
    E_{P^{*}}\left[\sum_{i=1}^{\infty}{\alpha_{i}(I_{A_{i} \cap F_{j}}-P(A_{i} \cap F_{j}))}\right] = E_{P^{*}}\left[\sum_{i \in I_{j}}{\alpha_{i}(I_{F_{j}}-P(F_{j}))}\right]
  \end{align}
	
  Next, we prove that $\sum_{i \in I_{j}}{\alpha_{i}}$ converges and, thus, all the equalities hold. Let $M_{j} = \bigcap{\{A_{i}: P(A_{i} \cap F_{j}) = P(F_{j})\}}$ and $N_{j} = \bigcup{\{A_{i}: P(A_{i} \cap F_{j}) = 0\}}$. Since $F_{j}$ is an atom and $P$ is countably additive, $P(M_{j}-N_{j}) = P(F_{j}) > 0$. Hence, $\exists w_{j} \in M_{j}-N_{j}$.	By construction, $\sum_{i=1}^{\infty}{\alpha_{i}I_{A_{i}}(w_{j})} = \sum_{i \in I_{j}}{\alpha_{i}}$. Since $(\alpha_{i},A_{i})_{i \in \mathbb{Z}^{+}}$ is in betting system 3, $\sum_{i=1}^{\infty}{\alpha_{i}I_{A_{i}}}$ converges pointwise on $\Omega$ and, thus, $\sum_{i \in I_{j}}{\alpha_{i}}$ converges. Applying equation \eqref{bet3cohEqn2} to \eqref{bet3cohEqn1},
  \begin{align*}
    E_{P^{*}}\left[\sum_{i=1}^{\infty}{\alpha_{i}(I_{A_{i}}-P(A_{i}))}\right] &= \sum_{j=1}^{n}{\left(\sum_{i \in I_{j}}{\alpha_{i}}\right) E_{P^{*}}\left[I_{F_{j}}-P(F_{j})\right]} = 0
  \end{align*}	
\end{proof}

Next, Example \ref{bet3Example} presents a system-3-coherent $P:\mathcal{C} \rightarrow [0,1]$ that admits no coherent extension to $\mathcal{F}(\mathcal{C})$. In this respect system $3$ differs from system $1$ and is similar to system $2$ (Example \ref{bet2Example}). Lemma \ref{P0Lemma} is useful for proving Example \ref{bet3Example}.

\begin{lemma}
  \label{P0Lemma}
  Let $\Omega = \{0,1\}^{\mathbb{Z}^{+}}$, $A_{i} = \{w \in \Omega: w_{i} = 1\}$, $\mathcal{C} = \{A_{i}: i \in \mathbb{Z}^{+}\}$ and $P^{0}: \sigma(\mathcal{C}) \rightarrow [0,1]$ be the only countably additive probability such that $I_{A_{i}}$ are i.i.d. and $P^{0}(A_{i})=2^{-1}$. For every $n \in \mathbb{Z}^{+}$, and betting portfolio $(\beta_{i},B_{i})_{i \in \mathbb{Z}^{+}}$, $P^{0}\left(\sum_{i=1}^{n}{\beta_{i}(I_{B_{i}}-2^{-1}) \geq 0}\right) \geq 2^{-1}$.
\end{lemma}

\begin{proof}
  Let $f: \Omega \rightarrow \Omega$ such that, for every $i \in \mathbb{Z}^{+}$, $f(w)_{i} = 1-w_{i}$.
	
  \begin{align*}
    \forall i \in \mathbb{Z}^{+}, \forall w \in \Omega, I_{A_{i}}(w)-2^{-1} = -(I_{A_{i}}(f(w))-2^{-1})
  \end{align*}

  Hence, for every $w \in \Omega$ such that $\sum_{i=1}^{n}{\beta_{i}(I_{B_{i}}(w)-2^{-1})} < 0$,
	
  \[ \sum_{i=1}^{n}{\beta_{i}(I_{B_{i}}(f(w))-2^{-1})} = -\sum_{i=1}^{n}{\beta_{i}(I_{B_{i}}(w)-2^{-1})} > 0 \]

  $P^{0}\left(\sum_{i=1}^{n}{\beta_{i}(I_{B_{i}}-2^{-1}) \geq 0}\right) \geq 2^{-1}$ follows from the symmetry of $P^{0}$.
\end{proof}
	
\begin{example}
  \label{bet3Example}
  Let $\Omega$ and $\mathcal{C}$ be as in Lemma \ref{P0Lemma}. Let $P(A_{i}) = 2^{-1}-4^{-i}$. $P$ is system-3-coherent but admits no system-3-coherent extension to $\mathcal{F}(\mathcal{C})$.
\end{example}

\begin{proof}
  First, we show that $P$ admits no system-3-coherent extension to $\mathcal{F}(\mathcal{C})$. Let $P^{a}$ be an arbitrary extension of $P$ to $\mathcal{F}(\mathcal{C})$. Observe that the image of $\mathcal{F}(\mathcal{C})$ through $P^{a}$ is an infinite set. It follows from Theorem \ref{bet3equiv} that $P^{a}$ is system-3-incoherent.

  Next, we prove that $P$ is system-$3$-coherent. Define $P^{0}$ such as in Lemma \ref{P0Lemma}. Let $(\gamma_{i},C_{i})_{i \in \mathbb{Z}^{+}}$ be a betting portfolio included under $P$ in betting system $3$. Also let $B_{n} = \sum_{i=1}^{n}{\gamma_{i}(I_{C_{i}}-P(C_{i}))}$. Let $f: \mathbb{Z}^{+} \rightarrow \mathbb{Z}^{+}$, where $f(i)$ is the integer such that $C_{i} = A_{f(i)}$. Let $M_{n}$ be the image of $\{1,\ldots,n\}$ through $f$,

  \begin{align*}
    P^{0}(B_{n} \geq 0)	&= P^{0}\left(\sum_{i=1}^{n}{\gamma_{i}(I_{A_{f(i)}}-P(A_{f(i)})) \geq 0}\right) = \\
			&= P^{0}\left(\sum_{j \in M_{n}}{\left(\sum_{i=1}^{n}{\gamma_{i}I_{\{j\}}(f(i))}\right)(I_{A_{j}}-P(A_{j}))} \geq 0 \right)
  \end{align*}
  
  For every $j \in M_{n}$, define $\alpha_{j,n} = \sum_{i=1}^{n}{\gamma_{i}I_{\{j\}}(f(i))}$.

  \begin{align*}
    P^{0}(B_{n} \geq 0) = P^{0}\left(\sum_{j \in M_{n}}{\alpha_{j,n}(I_{A_{j}}-2^{-1})} \geq \sum_{j \in M_{n}}{\alpha_{j,n}4^{-j}} \right)
  \end{align*}

  If $\max_{j \leq M_{n}}{\alpha_{j,n}} \leq 0$, then $\sum_{j \in M_{n}}{\alpha_{j,n}4^{-j}} \leq 0$. Hence, by Lemma \ref{P0Lemma},

  \begin{align*}
    P^{0}(B_{n} \geq 0) \geq P^{0}\left(\sum_{j \in M_{n}}{\alpha_{j,n}(I_{A_{j}}-2^{-1})} \geq 0 \right) \geq 2^{-1}
  \end{align*}	

  Otherwise, let $j^{*} = \arg\max_{j \leq M_{n}}{\alpha_{j,n}}$. Since $\sum_{j \in M_{n}}{\alpha_{j,n}4^{-j}} \leq \alpha_{j^{*},n}\sum_{j \in M_{n}}{4^{-j}}$ and also $ \alpha_{j^{*},n}\sum_{j \in M_{n}}{4^{-j}} \leq \alpha_{j^{*},n}3^{-1}$, conclude that

  \begin{align*}
    P^{0}(B_{n} \geq 0)	&\geq P^{0}\left(\sum_{j \in M_{n}}{\alpha_{j,n}(I_{A_{j}}-2^{-1})} \geq \alpha_{j^{*},n}3^{-1} \right) \geq	\\
			&\geq P^{0}\left(\sum_{j \in M_{n}}{\alpha_{j,n}(I_{A_{j}}-2^{-1})} \geq \alpha_{j^{*},n}3^{-1} \cap A_{j^{*}}\right) \geq\\
			&\geq P(A_{j^{*}}) P^{0}\left(\sum_{j \in M_{n}-\{j^{*}\}}{\alpha_{j,n}(I_{A_{j}}-2^{-1})} \geq -\alpha_{j^{*},n}(2^{-1}-3^{-1}) \right) \geq 4^{-1}
  \end{align*}

  The last inequality follows from Lemma \ref{P0Lemma} and $\alpha_{j^{*},n}4^{-1} > 0$. Thus, $\forall n \in \mathbb{Z}^{+}$, $P^{0}(B_{n} \geq 0) \geq 4^{-1}$. Hence, $P^{0}(w: B_{n} \geq 0 \text{ infinitely often}) \geq 4^{-1}$. Since $(\gamma_{i},C_{i})_{i \in \mathbb{Z}^{+}}$ is in betting system $3$ and $B_{\infty}$ is the balance of this portfolio, $B_{\infty}$ is defined for every $w \in \Omega$. Conclude that there exists $w \in \Omega$ such that $B_{\infty}(w) \geq 0$ and $P$ is system-3-coherent.
\end{proof}

\section{Betting system 2A}
\label{bet2A}

Consider that $\mathcal{C}$ is a field. Betting system $2A$ extends system $2B$. Hence, it follows from \cite{Beam} that, if $P$ is system-2A-coherent, then $P$ is a countably additive probability. Next, we characterize system-2A-coherence.

\begin{theorem}
  \label{bet2Aequiv}
  Let $\mathcal{C}$ be a field. $P$ is system-2A-coherent iff $P$ is a countably additive probability.
\end{theorem}

\begin{proof}
  Since system $2B$ is included in system $2A$, conclude from \cite{Beam} that the system-$2A$-coherence of $P$ implies that $P$ is a countably additive probability. The reverse implication is proved in Lemma \ref{bet2Acoh}.
\end{proof}

\begin{lemma}
  \label{bet2Acoh}
  Let $\mathcal{C}$ be a field and $P$ be a countably additive probability. For every betting portfolio $(\alpha_{i},A_{i})_{i \in \mathbb{Z}^{+}}$ in betting system $2A$, $E_{P^{*}}[\sum_{i=1}^{\infty}{\alpha_{i}(I_{A_{I}}-P(A_{i}))}] = 0$. Hence, $P$ is system-$2A$-coherent.
\end{lemma}

\begin{proof}	
  Since $(\alpha_{i},A_{i})_{i \in \mathbb{Z}^{+}}$ is in system $2A$, $\sum_{i=1}^{\infty}{|\alpha_{i}||I_{A_{i}}-P(A_{i})|}$ converges pointwise. Using the Monotone Convergence Theorem, conclude that

  \begin{align*}
    E_{P^{*}}\left[\sum_{i=1}^{\infty}{|\alpha_{i}||I_{A_{i}}-P(A_{i})|}\right]	&=	\sum_{i=1}^{\infty}{2|\alpha_{i}|P(A_{i})(1-P(A_{i}))}	 \leq \\
											&\leq	\sum_{i=1}^{\infty}{2|\alpha_{i}|\min\{P(A_{i}),1-P(A_{i})\}}									\leq	\\
											&\leq	2\sum_{i=1}^{\infty}{|\alpha_{i}||I_{A_{i}}-P(A_{i})}| < \infty
  \end{align*}

  Hence, using Fubini's Theorem,

  \begin{align*}
    E_{P^{*}}\left[\sum_{i=1}^{\infty}{\alpha_{i}(I_{A_{i}}-P(A_{i}))}\right] = \sum_{i=1}^{\infty}{E_{P^{*}}\left[\alpha_{i}(I_{A_{i}}-P(A_{i}))\right]} = 0
  \end{align*}
\end{proof}

Next, Example \ref{bet2ABExample} resolves a question posed by \cite{Beam}: Can every system-$2B$-coherent price assignment $P$ be extended coherently to a field?

\begin{example}
  \label{bet2ABExample}
  Let $(\Omega, \mathcal{C}, P)$ and $\delta$ be such as in Example \ref{bet2Example}. $P$ is system-$2A$ and system-$2B$-coherent but admits no system-$2A$ or system-$2B$-coherent extension to $\mathcal{F}(\mathcal{C})$.
\end{example}

\begin{proof}
  From Example \ref{bet2Example}, $P$ admits no countably additive extension to $\mathcal{F}(\mathcal{C})$. Hence by the characterization of coherence in \cite{Beam}, $P$ admits no system-2B or system-2A-coherent extension to $\mathcal{F}(\mathcal{C})$. 

  Next, consider an arbitrary betting portfolio $(\beta_{i},B_{i})_{i \in \mathbb{Z}^{+}}$ in betting system $2A$. For all $w \in \Omega$, $\sum_{i=1}^{\infty}{|\beta_{i}||I_{B_{i}}-P(B_{i})|} < \infty$. Take $w \notin A_{1}$,

  \begin{align*}
    \sum_{i=1}^{\infty}{|\beta_{i}|P(B_{i})} = \sum_{i=1}^{\infty}{|\beta_{i}||I_{B_{i}}(w)-P(B_{i})|} < \infty
  \end{align*}

  Thus, every betting portfolio in betting system $2A$ is also in betting system $2$. From Example \ref{bet2Example}, $P$ is system-$2A$-coherent. Since betting system $2A$ extends $2B$, $P$ is system-$2B$-coherent.
\end{proof}

Beam \cite{Beam} also questions: If $P$ is system-$2B$-coherent, can there exist a betting portfolio $(\beta_{i},B_{i})_{i \in \mathbb{Z}^{+}}$ which leads to uniform sure loss and such that the truncated balances $\sum_{i=n}^{\infty}{\alpha_{i}(I_{B_{i}}-P(B_{i}))}$ are pointwise convergent on $\Omega$ and uniformly bounded over $n \in \mathbb{Z}^{+}$ on $\Omega$? Example \ref{beamExample} provides an answer.

\begin{example}
  \label{beamExample}
  Let $(\Omega, \mathcal{C}, P)$, $(A_{i})_{i \geq 1}$ and $\delta$ be such as in Example \ref{bet2Example}. $P$ is system-2B-coherent and there exists a betting portfolio $(\beta_{i},B_{i})_{i \in \mathbb{Z}^{+}}$ which leads to uniform sure loss and such that the truncated balances $\sum_{i=n}^{\infty}{\beta_{i}(I_{B_{i}}-P(B_{i}))}$ are pointwise convergent on $\Omega$ and uniformly bounded over $n \in \mathbb{Z}^{+}$ on $\Omega$.
\end{example}

\begin{proof}
  Example \ref{bet2ABExample} shows thats $P$ is system-$2B$-coherent. Let $(\sigma_{i},S_{i})_{i \in \mathbb{Z}^{+}} = \left(\frac{(-1)^{i}}{i+1},A_{i}\right)_{i \in \mathbb{Z}^{+}}$. Next, let $\pi$ be the permutation of $\mathbb{Z}^{+}$ such that, $\pi(3i-2) = 4i-3$, $\pi(3i-1) = 4i-1$ and $\pi(3i) = 2i$. Also let $(\tau_{i},T_{i})_{i \in \mathbb{Z}^{+}} = \left(\frac{(-1)^{\pi(i)+1}}{\pi(i)+1},A_{\pi(i)}\right)_{i \in \mathbb{Z}^{+}}$. Finally, let $(\beta_{i},B_{i})_{i \in \mathbb{Z}^{+}}$ be such that $(\beta_{2i-1},B_{2i-1}) = (\sigma_{i},S_{i})$ and $(\beta_{2i},B_{2i}) = (\tau_{i},T_{i})$. 

  Since $\sum_{i=1}^{\infty}{\sigma_{i}(I_{S_{i}}-P(S_{i}))}$ and $\sum_{i=1}^{\infty}{\tau_{i}(I_{T_{i}}-P(T_{i}))}$ are pointwise convergent on $\Omega$, so is $\sum_{i=1}^{\infty}{\beta_{i}(I_B{_{i}}-P(B_{i}))}$ and

  \begin{align*}
    \sum_{i=1}^{\infty}{\beta_{i}(I_{B_{i}}-P(B_{i}))} = \sum_{i=1}^{\infty}{\sigma_{i}(I_{S_{i}}-P(S_{i}))} + \sum_{i=1}^{\infty}{\tau_{i}(I_{T_{i}}-P(T_{i}))}
  \end{align*}

  Hence, the truncated balances $\sum_{i=n}^{\infty}{\beta_{i}(I_{B_{i}}-P(B_{i}))}$ converge pointwise on $\Omega$. 

  Also, observe that, by construction, for every $w \in \Omega$, $\sum_{i=1}^{\infty}{\sigma_{i}I_{S_{i}}} = -\sum_{i=1}^{\infty}{\tau_{i}I_{T_{i}}}$. Thus, for every $w \in \Omega$,

  \begin{align*}
    \sum_{i=1}^{\infty}{\beta_{i}(I_{B_{i}}-P(B_{i}))}	&= -\sum_{i=1}^{\infty}{\sigma_{i}P(S_{i})} - \sum_{i=1}^{\infty}{\tau_{i}P(T_{i})} = \\
							&= -\sum_{i=1}^{\infty}{\frac{(-1)^{i}}{i+1}\left(\frac{1}{i+1}+\delta \right)} - \sum_{i=1}^{\infty}{\frac{(-1)^{\pi(i)+1}}{\pi(i)+1}\left(\frac{1}{\pi(i)+1}+\delta \right)} = \\
							&= -\delta \left(\sum_{i=1}^{\infty}{\frac{(-1)^{i}}{i+1}} - \sum_{i=1}^{\infty}{\frac{(-1)^{\pi(i)}}{\pi(i)+1}}\right) < 0
  \end{align*}

  The last equality does not depend on $w$ and, thus, $(\beta_{i},B_{i})_{i \in \mathbb{Z}^{+}}$ leads to uniform sure loss. It remains to show that the truncated balances $\sum_{i=n}^{\infty}{\beta_{i}(I_{B_{i}}-P(B_{i}))}$ are uniformly bounded over $n \in \mathbb{Z}^{+}$ on $\Omega$. Since $\sum_{i=1}^{\infty}{\beta_{i}I_{B_{i}}} = 0$,

  \begin{align*}
    \left|\sum_{i=n}^{\infty}{\beta_{i}(I_{B_{i}}-P(B_{i}))}\right| &\leq \left|\sum_{i=1}^{n-1}{\beta_{i}I_{B_{i}}}\right| + \left|\sum_{i=n}^{\infty}{\beta_{i}P(B_{i})}\right|
  \end{align*}

  Since $\sum_{i=n}^{\infty}{\beta_{i}P(B_{i})}$ is constant on $\Omega$ and $\sum_{i=1}^{\infty}{\beta_{i}P(B_{i})}$ converges, it remains to show that $\sum_{i=1}^{n-1}{\beta_{i}I_{B_{i}}}$ is uniformly bounded over $n \in \mathbb{Z}^{+}$ on $\Omega$. Let $n^{*} = \left\lfloor\frac{n}{2}\right\rfloor$.

  \begin{align*}
    \left|\sum_{i=1}^{n-1}{\beta_{i}I_{B_{i}}}\right| &\leq \left|\sum_{i=1}^{n^{*}}{\left(\sigma_{i}I_{S_{i}} + \tau_{i}I_{T_{i}}\right)}\right| + 1
  \end{align*}

  Observe that the image of $\{1,\ldots,n^{*}\}$ through $\pi$ is composed of the first $n_{e} = \left\lfloor \frac{n^{*}}{3} \right\rfloor$ even numbers and the first $n_{o} = n^{*}-n_{e}$ odd numbers. Hence, due to the symmetry of $\tau_{i}$ in relation to $\sigma_{i}$, 

  \begin{align*}
    \left|\sum_{i=1}^{n-1}{\beta_{i}I_{B_{i}}}\right|	&\leq \left|\sum_{i= \left\lceil\frac{n^{*}}{2}\right\rceil+1}^{n_{o}}{\frac{1}{2i}}\right| + \left|\sum_{i=n_{e}+1}^{\left\lceil\frac{n^{*}}{2}\right\rceil}{\frac{1}{2i+1}}\right| + 1 \leq \\
							&\leq \left|\sum_{i= \left\lfloor\frac{n}{4}\right\rfloor}^{\left\lceil \frac{n}{3}\right\rceil}{\frac{1}{2i}}\right| + \left|\sum_{i=\left\lfloor\frac{n}{6}\right\rfloor}^{\left\lceil\frac{n}{4}\right\rceil}{\frac{1}{2i+1}}\right| + 1
  \end{align*}

  Since both sums on the right side are constant on $\Omega$ and converge as $n \rightarrow \infty$, conclude that $\sum_{i=1}^{n-1}{\beta_{i}I_{B_{i}}}$ is uniformly bounded over $n \in \mathbb{Z}^{+}$ on $\Omega$. 
\end{proof}

\begin{corollary}
  Let $\mathcal{C}$ be a field and $P$ be system-$2A$-coherent. For every random variable $X$ which is $\sigma(\mathcal{C})$-measurable, there exists a betting portfolio in betting system 2A with balance $B$ which is a version of $X$ iff $E_{P^{*}}[X] = 0$.
\end{corollary}

\begin{proof}
  Lemma \ref{bet2Acoh} shows that $E_{P^{*}}[X] = 0$ is a necessary condition. The same construction as in Theorem \ref{bet2Balances} shows that it is also a sufficient condition.
\end{proof}

\section{The Space of Balances}
\label{spaceOfBalances}

Consider a betting system $\Pi$ which includes more betting portfolios than betting system $2$. Consider that $P$ is strongly $\Pi$-coherent. A balance $B$ is generated by $\Pi$ and $P$ if there exists a betting portfolio in $\Pi$ with balance $B$. What types of balances can be generated by $\Pi$ and $P$ but not by betting system $2$ and $P$? We restrict ourselves to the case in which $\Pi$ is a complete betting system, as follows:
 
\begin{definition}
  \label{completeSystem}
  $\Pi$ is a complete betting system if

  \begin{enumerate}[label=(\roman{*}), ref=(\roman{*})]
    \item It contains betting system $2$.
    \item It is contained by betting system $3$.
    \item If $(\alpha_{i},A_{i})_{i \in \mathbb{Z}^{+}}$ is in $\Pi$, then $(-\alpha_{i},A_{i})_{i \in \mathbb{Z}^{+}}$ is in $\Pi$.
    \item If $(\gamma_{i},C_{i})_{i \in \mathbb{Z}^{+}}$ and $(\beta_{i},B_{i})_{i \in \mathbb{Z}^{+}}$ are in $\Pi$, then $(\alpha_{i},A_{i})_{i \in \mathbb{Z}^{+}}$ with $(\alpha_{2n-1},A_{2n-1}) = (\beta_{n}, B_{n})$ and $(\alpha_{2n},A_{2n}) = (\gamma_{n}, C_{n})$ is in $\Pi$.
    \item Let $(\alpha_{i},A_{i})_{i \in \mathbb{Z}^{+}}$ be in $\Pi$. If $\pi$ is a permutation of $\mathbb{Z}^{+}$ and the balance and price of $(\alpha_{\pi(i)},A_{\pi(i)})_{i \in \mathbb{Z}^{+}}$ are defined, then $(\alpha_{\pi(i)},A_{\pi(i)})_{i \in \mathbb{Z}^{+}}$ is in $\Pi$. 
  \end{enumerate}
\end{definition}

It follows from definition that betting systems $2$ and $3$ are complete. We analyse complete systems according to the following alternative to coherence.

\begin{definition}
  \label{strongCoherence}
  $Z \in \mathcal{P}(\Omega)$ is a null set if, $\forall \delta > 0$, there exists $Z_{\delta} \in \mathcal{C}$ such that $Z \subset Z_{\delta}$ and $P(Z_{\delta}) < \delta$. $P$ is weakly incoherent in a given betting system if there exists a betting portfolio in that system, $(\alpha_{i},A_{i})_{i \in \mathbb{Z}^{+}}$ which leads to weak uniform sure loss. That is, there exists a null set $Z \in \mathcal{P}(\Omega)$, and $\epsilon > 0$, such that, $\forall w \notin Z$, $\sum_{i=1}^{\infty}{\alpha_{i}(I_{A_{i}}(w)-P(A_{i}))} \leq -\epsilon$. A probability assignment is strongly coherent if it is not weakly incoherent.
\end{definition}

We restrict ourselves to the case in which $\mathcal{C}$ is a field. In this case, Proposition \ref{strongCoherenceEquiv} shows that weak incoherence corresponds to $P^{*}$ almost sure loss.

\begin{proposition}
  \label{strongCoherenceEquiv}
  If $\mathcal{C}$ is a field, then $P$ is weakly incoherent in $\Pi$ iff there exists a betting portfolio $(\alpha_{i},A_{i})_{i \in \mathbb{Z}^{+}}$ in the system which leads to weak sure loss, i.e., for some $\epsilon > 0$, $\sum_{i=1}^{\infty}{\alpha_{i}(I_{A_{i}}(w)-P(A_{i}))} \leq -\epsilon$ a.s. $P^{*}$.
\end{proposition}

\begin{proof}
  Suppose there exists $\epsilon > 0$ and a betting portfolio $(\alpha_{i},A_{i})_{i \in \mathbb{Z}^{+}}$ in $\Pi$ such that $\sum_{i=1}^{\infty}{\alpha_{i}(I_{A_{i}}(w)-P(A_{i}))} \leq -\epsilon$ a.s. $P^{*}$. Let $B_{\epsilon} = \{\sum_{i=1}^{\infty}{\alpha_{i}(I_{A_{i}}(w)-P(A_{i}))} > -\epsilon\}$. Since $P^{*}(B_{\epsilon}) = 0$ and $\mathcal{C}$ is a field, it follows from \cite{Billingsley}[p.~186] that $B_{\epsilon}$ is a null set and $P$ is weakly incoherent. If $P$ is weakly incoherent, then there exists $\epsilon > 0$, a betting portfolio in $\Pi$ $(\alpha_{i},A_{i})_{i \in \mathbb{Z}^{+}}$ and a null set $Z$ such that $B_{\epsilon} = \{\sum_{i=1}^{\infty}{\alpha_{i}(I_{A_{i}}(w)-P(A_{i}))} > -\epsilon\} \subset Z$. Since $B_{\epsilon} \in \sigma(\mathcal{C})$ and $B_{\epsilon} \subset Z$, $P^{*}(B_{\epsilon}) = 0$. Hence, $\sum_{i=1}^{\infty}{\alpha_{i}(I_{A_{i}}(w)-P(A_{i}))} \leq -\epsilon$ a.s. $P^{*}$.
\end{proof}

The following Lemma is useful when proving weak incoherence. In this respect, it is similar to Lemma \ref{constantBet}.

\begin{lemma}
  \label{constantBet2}
  If $\mathcal{C}$ is a field, and there exist $c > 0$ and a betting portfolio $(\alpha_{i},A_{i})_{i\in \mathbb{Z}^{+}}$ in $\Pi$ with balance $B$ such that $B \leq -c$ a.s. $P^{*}$ or $B \geq c$ a.s. $P^{*}$, then $P$ is weakly $\Pi$-incoherent.
\end{lemma}

\begin{proof}
  If $B \leq -c$ a.s. $P^{*}$, the weak $\Pi$-incoherence of $P$ follows from Proposition \ref{strongCoherenceEquiv}. If $B \geq c$ a.s. $P^{*}$, by property $(iii)$ of the completeness of $\Pi$, $(-\alpha_{i},A_{i})_{i \in \mathbb{Z}^{+}}$ is in $\Pi$. The balance of this portfolio is $-B$ and $-B \leq -c$ a.s. $P^{*}$. Conclude that $P$ is weakly $\Pi$-incoherent.
\end{proof}

Theorem \ref{completesBalances} provides a characterization of the balances which are generated by $\Pi$ and a strongly $\Pi$-coherent $P$.

\begin{theorem}
  \label{completesBalances}
  If $\mathcal{C}$ is a field and $P$ is strongly $\Pi$-coherent, then every balance $B$ generated by $\Pi$ and $P$ is either a version of a balance generated by betting system $2$ and $P$ or such that $E_{P^{*}}[B]$ is not defined.
\end{theorem}

\begin{proof}
  From property $(ii)$ of the completeness of $\Pi$, if $X$ is a balance generated by $\Pi$, then it is a $\sigma(\mathcal{C})$-measurable random variable. From Theorem \ref{bet2Balances}, any $\sigma(\mathcal{C})$-measurable $X$ such that $E_{P^{*}}[X] = 0$ is a version of a balance of a betting portfolio in betting system $2$. Hence, from property (i), it remains to prove that, for every $\sigma(\mathcal{C})$-measurable $X$, such that $E_{P^{*}}[X] \neq 0$, $X$ is not a balance of a betting portfolio in $\Pi$. Assume $(\beta_{i},B_{i})_{i \in \mathbb{Z}^{+}}$ is a betting portfolio in $\Pi$ with balance $X$.
  
  Let $E_{P^{*}}[X] \neq 0$ be finite. Using Theorem \ref{bet2Balances}, there exists a betting portfolio in betting system $2$, $(\gamma_{i},C_{i})_{i \in \mathbb{Z}^{+}}$, such that $\sum_{i=1}^{\infty}{\gamma_{i}(I_{C_{i}}-P(C_{i}))} = E_{P^{*}}[X]-X$ a.s. $P^{*}$. By property $(i)$ of completeness, $(\gamma_{i},C_{i})_{i \in \mathbb{Z}^{+}}$ is in $\Pi$. By property $(iv)$ of completeness, $(\alpha_{i},A_{i})_{i \in \mathbb{Z}^{+}}$ with $(\alpha_{2n-1},A_{2n-1}) = (\beta_{n}, B_{n})$ and $(\alpha_{2n},A_{2n}) = (\gamma_{n}, C_{n})$ is in $\Pi$. $\sum_{i=1}^{\infty}{\alpha_{i}(I_{A_{i}}-P(A_{i}))} = E_{P^{*}}[X]$ a.s. $P^{*}$. A contradiction follows from Lemma \ref{constantBet2}.
  
  Let $E_{P^{*}}[X] = -\infty$. Hence, there exists $c < 0$ and $k < 0$ such that $E_{P^{*}}[-X\cdot I_{X > c}+k] = 0$. Using Theorem \ref{bet2Balances} and property $(i)$ of completeness, there exists a betting portfolio $(\gamma_{i},C_{i})_{i \in \mathbb{Z}^{+}}$ in $\Pi$ with balance equal to $-X\cdot I_{X > c}+k$ a.s. $P^{*}$. By property $(iv)$ of completeness, $(\alpha_{i},A_{i})_{i \in \mathbb{Z}^{+}}$ with $(\alpha_{2n-1},A_{2n-1}) = (\beta_{n}, B_{n})$ and $(\alpha_{2n},A_{2n}) = (\gamma_{n}, C_{n})$ is in $\Pi$. By construction, $\sum_{i=1}^{\infty}{\alpha_{i}(I_{A_{i}}-P(A_{i}))} = X\cdot I_{X \leq c}+k$ a.s. $P^{*}$. A contradiction follows from Lemma \ref{constantBet2}.
  
  Let $E_{P^{*}}[X] = \infty$. By property $(iii)$ of completeness, $(-\beta_{i},B_{i})_{i \in \mathbb{Z}^{+}}$ is in $\Pi$ and also $E_{P^{*}}[\sum_{i=1}^{\infty}{-\beta_{i}(I_{B_{i}}-P(B_{i}))}] = -\infty$, a contradiction.
\end{proof}

Whether there exists $B$ generated by $\Pi$ and $P$ such that $E_{P^{*}}[B]$ is not defined is unresolved. This is the object of the following conjecture:

\begin{conjecture}
  \label{balancesConj}
  If $\mathcal{C}$ is a field and $P$ is strongly $\Pi$-coherent, then every balance generated by $\Pi$ and $P$ is a version of a balance generated by betting system $2$ and $P$.
\end{conjecture}

\section{Alternatives to Coherence}
\label{altCoherence}

The previous Sections use the definition of coherence and strong coherence. \cite{Adams} also considers the following definition:

\begin{definition}
  \label{rationality}
  $P$ is irrational in a given betting system if there exists a betting portfolio included in that system, $(\alpha_{i},A_{i})_{i \in \mathbb{Z}^{+}}$, which leads to sure loss. That is, $\forall w \in \Omega$, $\sum_{i=1}^{\infty}{\alpha_{i}(I_{A_{i}}(w)-P(A_{i}))} < 0$. $P$ is rational if it is not irrational.
\end{definition}
 
Every incoherent $P$ is also weakly incoherent and irrational. In the betting system described in \cite{deFinetti2} balances can assume only a finite number of values. In this case, the three definitions are equivalent. Next, we show that, when $\mathcal{C}$ is a field, the three conditions are equivalent in betting systems $1$, $2$, $2B$, $2A$ and $3$.

\begin{theorem}
  \label{coherenceEquiv}
  Let $\mathcal{C}$ be a field. In betting systems $1$, $2$, $2B$, $2A$ and $3$ the following statements are equivalent:
  \begin{itemize}
    \item $P$ is coherent.
    \item $P$ is strongly coherent.
    \item $P$ is rational.
  \end{itemize}
\end{theorem}

\begin{proof}
  If a betting portfolio leads to uniform sure loss, than it leads to sure loss and weak uniform sure loss. Hence, independently of the betting system under consideration, if $P$ is rational or if $P$ is strongly coherent, then $P$ is coherent.

  Consider that $P$ is system-1-coherent. Hence, \cite{deFinetti2} implies that $P$ is a finitely additive probability. \cite{Kadane2}[p.~25] shows that, for every betting portfolio $(\alpha_{i},A_{i})_{i \in \mathbb{Z}^{+}}$ included in system $1$, $E[\sum_{i=1}^{\infty}{\alpha_{i}(I_{A_{i}}-P(A_{i}))}] = 0$. Thus, $P$ is strongly coherent and rational. 

  Similarly, if $P$ is system-2-coherent, system-2B-coherent or system-2A-coherent, Theorems \ref{bet2equiv} and \ref{bet2Aequiv} imply that $P$ is a countably additive probability. Hence,  Lemmas \ref{bet2coh} and \ref{bet2Acoh} imply that, for every betting portfolio $(\alpha_{i},A_{i})_{i \in \mathbb{Z}^{+}}$ included in system $2$, $2B$ or $2A$, $E_{P^{*}}[\sum_{i=1}^{\infty}{\alpha_{i}(I_{A_{i}}-P(A_{i}))}] = 0$. Thus, $P$ is strongly coherent and rational.

  Finally, if $P$ is system-3-coherent, Theorem \ref{bet3equiv} implies that $P$ is a countably additive probability and $\Omega$ admits a finite partition in $\mathcal{C}$-measurable atoms. Hence, Lemma \ref{bet3coh} implies that, for every betting portfolio $(\alpha_{i},A_{i})_{i \in \mathbb{Z}^{+}}$ included in system $3$, $E_{P^{*}}[\sum_{i=1}^{\infty}{\alpha_{i}(I_{A_{i}}-P(A_{i}))}] = 0$. Thus, $P$ is strongly coherent and rational.
\end{proof}

\cite{Adams} presents as an open question to determine a condition equivalent to the system-2A-rationality of $P$ when $\mathcal{C}$ is a field.

\begin{proposition}
  Let $\mathcal{C}$ be a field. $P$ is system-2A-rational iff $P$ is a countably additive probability.
\end{proposition}

\begin{proof}
  Follows from Theorem \ref{bet2Aequiv} and \ref{coherenceEquiv}.
\end{proof}

\section{Conclusions}
\label{conclusions}

We present characterizations of coherence for the betting systems under consideration when $\mathcal{C}$ is a field. Table $1$ summarizes the results which follow from \cite{deFinetti2}, \cite{Beam} and Theorems \ref{bet2equiv}, \ref{bet3equiv} and \ref{bet2Aequiv}. Theorems \ref{bet2Aequiv} and \ref{coherenceEquiv} solve a question presented in \cite{Adams}.

If $\mathcal{C}$ is a field, we also conclude from Theorem \ref{coherenceEquiv} that, for all of the betting systems under consideration, coherence is equivalent to strong coherence and rationality.

\begin{table}[ht]
  \centering
  \begin{tabular}{|c|c|c|}
    \hline
    System	& Allows portfolios with\ldots						& $P$ is coherent on a field iff		\\
    \hline
    1		& finite number of bets  						& $P$ is a finitely additive probability	\\ [0.25em]
    2		& invariance under permutations 					& $P$ is a countably additive probability	\\ [0.25em]
    2B		& $\sum_{i=1}^{\infty}{|\alpha_{i}(I_{A_{i}}-P(A_{i}))|} < M$ 	& $P$ is a countably additive probability	\\ [0.25em]
    2A		& $\sum_{i=1}^{\infty}{|\alpha_{i}(I_{A_{i}}-P(A_{i}))|} < \infty$	& $P$ is a countably additive probability	\\ [0.25em]
    3		& price and balance defined 						& $\Omega$ is partitioned in finite atoms and	\\
		&									& $P$ is a countably additive probability 	\\
    \hline
  \end{tabular}
  \label{tab:summary}
  \caption{Characterizations of Coherence.}
\end{table}

We also consider whether every coherent price assignment on an arbitrary $\mathcal{C}$ can be extended to a coherent assignment on $\mathcal{F}(\mathcal{C})$. While \cite{Heath} proves this result in betting system $1$, Examples \ref{bet2Example}, \ref{bet2ABExample} and \ref{bet3Example} present counter-examples in systems $2$, $2B$, $2A$ and $3$.

Next, for each betting system, we consider a fixed coherent $P$. We characterize which $\sigma(\mathcal{C})$-measurable functions had versions in the space of balances generated by $P$ and the betting system under consideration when $\mathcal{C}$ is a field. Table $2$ summarizes this result and is obtained from Theorems \ref{bet2Balances} and \ref{bet3equiv}. Tables $1$ and $2$ show that, although coherence is equivalent in betting systems 2 and $2B$, the space of balances generated by system $2$ is larger.

\begin{table}[h]
  \centering
  \begin{tabular}{|c|c|}
    \hline
    System	& $\exists$ a version of $X$ in the Space of Balances iff		\\
    \hline
    1		& $X$ is a simple $\mathcal{C}$-measurable function and $E_{P}[X]=0$	\\ [0.25em]
    2 		& $E_{P^{*}}[X]= 0$							\\ [0.25em]
    2B		& $X$ is a bounded function and $E_{P^{*}}[X]= 0$			\\ [0.25em]
    2A		& $E_{P^{*}}[X]= 0$							\\ [0.25em]
    3 		& $X$ is a simple function and $E_{P^{*}}[X]= 0$			\\
    \hline
  \end{tabular}
  \label{tab:summary2}
  \caption{Characterization of Space of Balances on a field.}
\end{table}

Finally, we characterize what kinds of balances can be generated by an arbitrary betting portfolio $\Pi$ when $P$ is strongly $\Pi$-coherent. We restrict ourselves to $\Pi$ being a complete betting system. In this case, if $\mathcal{C}$ is a field, every balance $B$ generated by $\Pi$ and $P$ is either a version of a balance generated by betting system $2$ and $P$ or such that $E_{P^{*}}[B]$ is not defined. It remains an open question if every balance $B$ generated by $\Pi$ and $P$ is a version of a balance generated by betting system $2$ and $P$.

\bibliographystyle{amsplain}
\bibliography{onBets}
\end{document}